\documentclass[a4paper,reqno,12pt]{amsart}
\usepackage{verbatim}
\usepackage{amssymb,amsmath,amsthm}
\usepackage{amsfonts}
\usepackage{array}
\usepackage{xcolor}

\def\B3{\overline{B}_3}
\def\p3{p_{-3}}
\def\op{\overline{p}}
\newcommand{\legendre}[2]{\genfrac{(}{)}{}{}{#1}{#2}}

\newtheorem{theorem}{Theorem}[section]
\newtheorem{lemma}{Lemma}[section]

\newtheoremstyle{remark}
    {\dimexpr\topsep/2\relax} 
    {\dimexpr\topsep/2\relax} 
    {}          
    {}          
    {\bfseries} 
    {.}         
    {.5em}      
    {}          

\theoremstyle{remark}
\newtheorem{remark}{Remark}[section]

\makeatletter
\@namedef{subjclassname@2020}{%
  \textup{2020} Mathematics Subject Classification}
\makeatother

\begin{document}

\title[]{%
Arithmetic properties \\ of MacMahon-type sums of divisors}

\dedicatory{In grateful recognition of Krishnaswami Alladi's valued service as Editor--in--Chief of the {\it Ramanujan Journal} for more than a quarter century. Thanks to Alladi's leadership and inclusive spirit, the mathematics of Srinivasa Ramanujan has flourished within our community.}

\author{James A. Sellers}
\address{Department of Mathematics and Statistics, 
University of Minnesota Duluth, Duluth, MN 55812, USA}
\email{jsellers@d.umn.edu}

\author{Roberto Tauraso}
\address{Dipartimento di Matematica, 
Università di Roma ``Tor Vergata'', 00133 Roma, Italy}
\email{tauraso@mat.uniroma2.it}



\begin{abstract} In this paper, we prove several new infinite families of Ramanujan--like congruences satisfied by the coefficients of the generating function $U_t(a, q)$ which is an extension of  MacMahon's generalized sum-of-divisors function. 
As a by-product, we also show that, for all $n\geq 0$, $\B3(15n+7)\equiv 0  \pmod{5}$  where $\B3(n)$ is the number of almost $3$-regular overpartitions of $n$.
\end{abstract}

\maketitle

\section{Introduction}

\smallskip
\noindent
In \cite{AAT2}, Amdeberhan, Andrews, and 
the second author introduced the generating function
\begin{equation} \label{defofU} 
U_t(a, q):=\sum_{1\leq n_1<n_2<\cdots<n_t}\frac{q^{n_1+n_2+\cdots+n_t}}{\prod_{k=1}^t(1+aq^{n_k}+q^{2n_k})}=\sum_{n\geq0} MO(a, t; n)q^n.
\end{equation}
The particular case $U_t(-2, q)$ was initially studied by P. A. MacMahon \cite{MacMahon} as a generalization of the sum-of-divisors function,
$$U_1(-2, q)=\sum_{n\geq 1} \frac{q^n}{(1-q^n)^2}=\sum_{n\geq 1} \sigma(n) q^n.$$
The function $U_t(-2, q)$ has received much attention in recent years; see \cite{AAT1, AOS, Andrews-Rose, OS} for more details about the results that have already been proven. 

In \cite{AAT2}, the authors investigated $U_t(a, q)$ for the other cases when the roots of $1+ax+x^2=0$ are also roots of unity. This occurs precisely for $a=0, \pm1, \pm2$. As such, in this note we will continue the study of $U_t(a, q)$ for this set of values of $a$, focusing our attention on arithmetic properties  satisfied by the coefficients $MO(a, t; n)$.

From  Theorem 2.1 and Appendix 2 in \cite{AAT2}, we have that $U_t(a, q)$ can be written
as the product of two generating functions:
\begin{equation}\label{idofU} 
U_t(a, q)=\sum_{k\geq 0} b_k(a) q^k \cdot \sum_{n\geq 0} c_{n}(a, t) q^{\binom{n+1}{2}}
\end{equation}
where
$$\sum_{k\geq 0} b_k(a) q^k=\prod_{n=1}^{\infty}\frac{1}{(1+aq^n+q^{2n})(1-q^n)}$$
and
$$c_n(a, t)=(-1)^{n-t}[z^n]\frac{z^t(1+z)}{(1+az+z^2)^{t+1}}$$
where $[z^n]\left( \sum_{k\geq 0} h_kq^k \right) = h_n$, the coefficient of $z^n$ in the given power series.

Letting  $a=0, \pm1, \pm2$, from \eqref{idofU} (see \cite[Section 5]{AAT2}), we obtain explicit formulas of the prefactor $\sum_{k\geq 0} b_k(a) q^k$ and the coefficient $c_{n}(a, t)$.  Such formulas will be pivotal in our work.

\medskip 

\begin{center}
\begin{tabular}{|c|c|c|}
\hline\rule[-3mm]{0mm}{8mm}
$a$ & $\sum_{k\geq 0} b_k(a) q^k$ & $c_{n}(a, t)$\\
\hline\rule[-6mm]{0mm}{15mm}
1 & $\displaystyle \frac{1}{f_3}$ 
& $\displaystyle\sum_{k=0}^{n} \frac{(-1)^{n+k} (2n+1)\binom{n+k+1}{2k+1}}{n+k+1}\binom{k}{t}  3^{k-t}$\\
\hline\rule[-6mm]{0mm}{15mm}
$-1$ & $\displaystyle\frac{f_2 f_3}{f_1^2 f_6}$ 
& $\displaystyle\sum_{k=0}^n \frac{(-1)^{n+k}(2n+1)\binom{n+k+1}{2k+1}}{n+k+1}\binom{k}{t}$\\
\hline\rule[-6mm]{0mm}{15mm}
2 & $\displaystyle \frac{f_1}{f_2^2}$ 
& $\displaystyle\binom{n+t}{2t}$\\
\hline\rule[-6mm]{0mm}{15mm}
$-2$ & $\displaystyle\frac{1}{f_1^3}$ 
& $\displaystyle (-1)^{n+t} \frac{2n+1}{2t+1} \binom{n+t}{2t}$\\
\hline\rule[-6mm]{0mm}{15mm}
0 & $\displaystyle \frac{f_2}{f_1 f_4}$ 
& $\displaystyle (-1)^{n+\lfloor (n+t)/2\rfloor} \binom{\lfloor (n+t)/2\rfloor}{t}$\\
\hline
\end{tabular}
\end{center}\medskip
where $f_r=(q^r;q^r)_{\infty} = \prod_{k\geq 1}(1-q^{rk})$.

Our overarching goal in this work is to prove new Ramanujan--like congruences satisfied by the functions $MO(a, t; AN+B)$ for various values of $a$, $t$, $A$, and $B$, thus extending the set of known results that already appear in the literature.  In the next section, we provide a list of such known results along with some obvious congruence relationships between some of these functions.  We then prove numerous new results in the subsequent sections which are based on the value of the parameter $a$.

\section{Known arithmetic properties}
\label{sec:known_properties}

We begin by highlighting a relatively obvious set of congruence results.  
\begin{theorem}
\label{overarching2}
For any $t$, 
$$
U_t(-1, q) \equiv U_t(1, q)  \pmod{2}   \text{\ \ \ and\ \ \ }
U_t(-2, q) \equiv U_t(0, q) \equiv U_t(2, q) \pmod{2}.
$$
\end{theorem}
\begin{proof}
This follows immediately from \eqref{defofU}.
\end{proof}
\begin{theorem}
\label{overarching3}
For any $t$, 
$$
U_t(-1, q) \equiv U_t(2, q) \pmod{3} \text{\ \ \ and\ \ \ } U_t(1, q) \equiv U_t(-2, q) \pmod{3}.
$$
\end{theorem}
\begin{proof}
This also follows immediately from \eqref{defofU}.
\end{proof}

Next, we focus our attention on past results which already appear in the literature.  The case $a=-2$ has been investigated in \cite{AAT1} and \cite{AOS}. 
In \cite{AAT1}, Amdeberhan, Andrews, and 
the second author discovered that, for all $N\geq 0$,
\begin{align*}
&MO(-2, 2; 5N+2)\equiv 0\pmod 5 \\
&MO(-2, 3; 7N+3)\equiv MO(-2, 3; 7N+5)\equiv 0\pmod 7.
\end{align*}
They also conjectured that, for all $N\geq 0$,
\begin{equation}\label{AATConjecture}
MO(-2, 10; 11N+7)\equiv 0\pmod{11}.
\end{equation}
This has since been confirmed by Amdeberhan, Ono, and Singh 
\cite[Theorem 1.5]{AOS}.
Morever, in \cite[Theorem 1.6]{AOS}, the authors showed that, for all $J\geq 0$ and $N\geq 0$,
\begin{align*}
&MO(-2, 3J+2; 3N+1)\equiv MO(-2, 3J+2; 3N+2)\equiv  0\pmod{3},\\
&MO(-2, 11J+10; 11N+7)\equiv  0\pmod{11},\\
&MO(-2, 17J+16; 17N+15)\equiv  0\pmod{17}.
\end{align*}
Amdeberhan, Andrews, and 
the second author also considered arithmetic results for other values of $a$ besides $-2$.  As an example, they proved \cite[Theorem 2.2]{AAT2}  that, for all $N\geq 0$ and for all $t$,
\begin{equation}
\label{1_t_3n2}
MO(1, t; 3N+2)=0.
\end{equation}
Moreover, by \cite[Theorem 2.2]{AAT2}, for all $N\geq 0$,
\begin{equation}
MO(1, 3; 3N+1) \equiv 0 \pmod{3}.\label{1_3_3n1}
\end{equation}
With this set of results in hand, we now move ahead to prove additional infinite families of congruences satisfied by these functions.  

\section{The Case $a=1$}

From the table given in the introduction we have that
$$\sum_{k\geq 0} b_k(1) q^k=\frac{1}{f_3}=\sum_{k\geq 0} p(k) q^{3k}$$
where $p(k)$ is the number of integer partitions of $k$. 
Therefore
\begin{align}\label{idp1}
MO(1, t; N)=\sum_{3k+\binom{n+1}{2}=N}p(k)\cdot c_n(1, t)
\end{align}
where
\begin{align*}
c_n(1, t)&=\sum_{k=0}^{n} \frac{(-1)^{n+k} (2n+1)\binom{n+k+1}{2k+1}}{n+k+1}\binom{k}{t}  3^{k-t}\\
&=\sum_{k=0}^{n}(-1)^{n+k} \left(\binom{n+k+1}{2k+1}+\binom{n+k}{2k+1}\right)\binom{k}{t} 3^{k-t}.
\end{align*}

We already noted that $MO(1, t; 3N+2) = 0$ identically for all $t$.  Next, we significantly extend the congruence in \eqref{1_3_3n1}.  

\begin{theorem} \label{MO3n+1v2} For all $J\geq 0$ and $N\geq 0$,
$$MO(1, 3J+3; 3N+1) \equiv MO(1, 3J+2; 3N+1) \equiv  0 \pmod 3.$$
\end{theorem}

\begin{proof} Since by \eqref{idp1},
$$MO(1, t; 3N+1)=\sum_{3k+\binom{n+1}2=3N+1}p(k)\cdot c_n(1, t),$$
it suffices to check the case when
$\binom{n+1}{2}\equiv 1\pmod3$, that is $n=3m+1$ for some integer $m$, and show that 
$$c_{3m+1}(1, t)=\sum_{k=t}^{3m+1} (-1)^{3m+k+1}\left(\binom{3m+k+2}{2k+1}+\binom{3m+k+1}{2k+1}\right)\binom{k}{t}  3^{k-t}$$
is divisible by $3$.
Actually, each summand is immediately divisible by $3$ due to the term $3^{k-t}$ when $k>t$. For $k=t$, replacing this value to compute the corresponding term, we obtain
\begin{align*}
(-1)^{n+t}&\left(\binom{3m+t+2}{2t+1}+\binom{3m+t+1}{2t+1}\right)\binom{t}{t}  3^{t-t}\\
&= (-1)^{n+t}\frac{3(2m+1)}{2t+1}\binom{3m+t+1}{2t} 
\end{align*}
which is divisible by $3$ as long as $2t+1\not \equiv 0\pmod 3$, that is $t \not \equiv 1 \pmod{3}$. Our result follows.
\end{proof}

The next result provides a natural strengthening of Theorem \ref{MO3n+1v2} by increasing the modulus from 3 to 9.
 
\begin{theorem} \label{MO3n+1v3} For all $J\geq 0$ and $N\geq 0$,
$$MO(1, 3J+2; 3N+1)\equiv 0 \pmod{9}.$$
\end{theorem}

\begin{proof} 
As in the previous proof,
by \eqref{idp1}, it suffices to show that
$$c_{3m+1}(1, t)=\sum_{k=t}^{3m+1} (-1)^{n+k}\left(\binom{3m+k+2}{2k+1}+\binom{3m+k+1}{2k+1}\right)\binom{k}{t}  3^{k-t}$$
is divisible by $9$ when $t=3J+2$.

Each summand
\begin{align*}
(-1)^{n+k}&\left(\binom{3m+k+2}{2k+1}+\binom{3m+k+1}{2k+1}\right)\binom{k}{t}  3^{k-t}\\
&= (-1)^{n+k}\frac{3(2m+1)}{2k+1}\binom{3m+k+1}{2k} \binom{k}{t}  3^{k-t}
\end{align*}
is immediately divisible by $9$ due to the term $3^{k-t}$ for $k>t+1$. Let us consider the cases $k=t$ and $k=t+1$.

\noindent If $k=t=3J+2$ then the term is
\begin{align*}
(-1)^{n+k}&\frac{3(2m+1)}{6J+5}\binom{3(m+J+1)}{6J+4}\\
&=(-1)^{n+k}\frac{9(2m+1)(m+J+1)}{(6J+5)(6J+4)}\binom{3(m+J+1)-1}{6J+3}\equiv 0 \pmod{9}.
\end{align*}

\noindent If $k=t+1=3J+3$ then the term is
\begin{align*}
(-1)^{n+k}&\frac{3(2m+1)}{6J+7}\binom{3m+3J+4}{6J+6}(3J+3)\cdot 3\equiv 0 \pmod{9}.
\end{align*}

Combining all of the above facts yields our result. 
\end{proof}

The next result is related to the well-known congruences of Ramanujan \cite{R21}:   For all $N\geq 0$, 
\begin{align*}
p(5N+4) &\equiv 0 \pmod{5},\\
p(7N+5) &\equiv 0 \pmod{7},\\
p(11N+6) &\equiv 0 \pmod{11}.
\end{align*}

\begin{theorem} \label{MO5n+2}  For all $J\geq 0$ and all $N\geq 0$, 
\begin{align*}
MO(1, 5J+4; 5N+2) &\equiv 0 \pmod{5},\\
MO(1, 7J+6; 7N+1) &\equiv 0 \pmod{7},\\
MO(1, 11J+10; 11N+7) &\equiv 0 \pmod{11}.
\end{align*}
\end{theorem}

\begin{proof} We prove the first congruence. The other two follow in a similar way. 
If $t=5J+4$ then, by Lucas's theorem (see for instance \cite{Fine}), $\binom{k}{t}\equiv 0\pmod 5$ when $k\not \equiv 4\pmod 5$. Similarly, if $k\equiv 4\pmod 5$ then 
$$\binom{n+k+1}{2k+1}+\binom{n+k}{2k+1}\equiv 0\pmod 5$$
when $n\not \equiv 0,4\pmod 5$. Hence, if $c_n(1, 5J+4)$ is not divisible by $5$ then $n\equiv 0$ or  $n\equiv 4 \pmod{5}$ 
which implies that $\binom{n+1}2$ is divisible by $5$.
Thus by \eqref{idp1},
$$MO(1, 5J+4; 5N+2)=\sum_{3k+\binom{n+1}2=5N+2}p(k)\cdot c_n(1, 5J+4) \equiv 0 \pmod{5}$$
because when $\binom{n+1}2$ is divisible by $5$ then $3k\equiv 2\pmod{5}$, that is $k\equiv 4\pmod{5}$. In that case, $p(k)\equiv 0\pmod{5}$ by the first of Ramanujan's congruences recalled above.
\end{proof}

\section{The Case $a=-1$}

From the table given in the introduction we have that
$$\sum_{k\geq 0} b_k(-1) q^k=\frac{f_2 f_3}{f_1^2 f_6}=\sum_{k\geq 0} \B3(k) q^{k}$$
where $\B3(k)$ is the number of almost $3$-regular overpartitions of $k$  \cite{Ballantine-Merca}. 
Therefore
\begin{align}\label{idm1}
MO(-1, t; N)=\sum_{k+\binom{n+1}{2}=N}\B3(k)\cdot c_n(-1, t)
\end{align}
where
$$
c_n(-1, t)=(-1)^{n-t}[z^n]\frac{z^t(1+z)}{(1-z+z^2)^{t+1}}=\sum_{k=0}^n \frac{(-1)^{n+k}(2n+1)\binom{n+k+1}{2k+1}}{n+k+1}\binom{k}{t}.
$$
It is shown in \cite[Corollary 1.5]{Ballantine-Merca}, that for all $k\geq 0$,
$$\B3(3k+1)\equiv 0 \pmod{2}\qquad\text{and}\qquad \B3(3k+2)\equiv 0 \pmod{4}.$$
Moreover, by the proof given in the Appendix at the end of the paper, for all $k\geq 0$,
$$\B3(15k+7)\equiv 0\pmod{5}.$$

Equipped with these facts, we now proceed to prove several new arithmetic properties satisfied in the case when $a=-1$. 

\begin{theorem} \label{MOm1}  For all $t\geq 1$ and all $N\geq 0$,
$$MO(-1, t; 3N+2) \equiv 0 \pmod{2}.$$
\end{theorem}

\begin{proof} 
This follows from \eqref{1_t_3n2} and Theorem \ref{overarching2}. 
\end{proof}

The above result can be extended modulo $4$ for 
an infinite family of values of $t$.

\begin{theorem} \label{MOm14}  For all $J\geq 2$ and all $N\geq 0$,
$$MO(-1, 2^J-2; 3N+2) \equiv 0 \pmod{4}.$$
\end{theorem}

\begin{proof} By  \eqref{idm1}, 
$$MO(-1, 2^J-2; 3N+2)=\sum_{k+\binom{n+1}{2}=3N+2}\B3(k)\cdot c_n(-1, 2^J-2)\equiv 0 \pmod{4}$$
because

1) if $n\equiv 0,2 \pmod{3}$ then $\binom{n+1}{2}\equiv 0\pmod{3}$ which implies that $k\equiv 2\pmod{3}$ and $\B3(k)\equiv 0\pmod{4}$. 

2) if $n\equiv 1 \pmod{3}$ then $\binom{n+1}{2}\equiv 1\pmod{3}$ which implies that $k\equiv 1\pmod{3}$ and $\B3(k)\equiv 0\pmod{2}$  and $c_n(-1, 2^J-2)\equiv 0\pmod{2}$.

Indeed, the last congruence holds due to the fact that
\begin{align*}
 \frac{z^{2^J-2}(1+z)}{(1-z+z^2)^{2^J-1}}&=\frac{z^{2^J-2}(1+z)^{2^J}(1+z^3)}{(1+z^3)^{2^J}}\\
 &\equiv \frac{z^{2^J-2}(1+z^{2^J})(1+z^3)}{1+z^{3\cdot 2^J}} \pmod{2}\\
 &\equiv \frac{z^{2^J-2}+z^{2^j+1}+z^{2^{J+1}-2}+z^{2^{J+1}+1}}{1+z^{3\cdot 2^J}} \pmod{2}.
 \end{align*}
Thus, since in the numerator each exponent is not congruent to $1 \pmod{3}$, 
we may conclude that if $n\equiv 1\pmod{3}$ then $c_n(-1, 2^J-2)\equiv 0\pmod{2}$.
\end{proof}

We next prove an isolated congruence modulo 5 which was rather unexpected.

\begin{theorem} \label{MOm1m5}  For all $N\geq 0$,
$$MO(-1, 3; 15N+13) \equiv 0 \pmod{5}.$$
\end{theorem}

\begin{proof} We have that
$$c_n(-1, 3)=(-1)^{n-1}[z^n]\frac{z^3(1+z)}{(1-z+z^2)^{4}}.$$
Moreover
\begin{align*}
 \frac{z^3(1+z)}{(1-z+z^2)^{4}}&=\frac{z^3(1+z)(1-z+z^2)}{(1-z+z^2)^{5}}
 \equiv \frac{z^3(1+z^3)}{1-z^5+z^{10}}\\
 &=\frac{z^3(1+z^3)(1+z^5)}{1-z^{15}}
 \equiv \frac{z^3+z^8+z^6+z^{11}}{1-z^{15}} \pmod{5}.
 \end{align*}
Thus, it follows that if $c_n(-1, 3)\not\equiv 0\pmod{5}$ then $n\equiv 3,6,8,11\pmod{15}$ and therefore $\binom{n+1}{2}\equiv 6\pmod{15}$.
Hence, by  \eqref{idm1},
$$MO(-1, 3; 15N+13)=\sum_{k+\binom{n+1}{2}=15N+13}\B3(k)\cdot c_n(-1, 3)\equiv 0 \pmod{5}$$
because $\B3(k)\equiv 0\pmod{5}$ when $k\equiv 13-6=7 \pmod{15}$. 
\end{proof}

We close this section by noting the following family of congruences modulo 3.  
\begin{theorem} \label{MOneg1mod3}  For all $J\geq 0$ and all $N\geq 0$,
\begin{align*}
MO(-1, 9J+4; 27N+9) &\equiv 0 \pmod{3},\\
MO(-1, 9J+7; 27N) &\equiv 0 \pmod{3}.
\end{align*}
\end{theorem}
\begin{proof}
This result follows from Theorem \ref{MO2} below and Theorem \ref{overarching3}.
\end{proof}

\section{The Case $a=2$} 

From the table given in the introduction we have that
$$\sum_{k\geq 0} b_k(2) q^k=\frac{f_1}{f_2^2}=\sum_{k\geq 0} (-1)^k pod(k) q^{k}$$
because if we replace $q$ by $-q$, then $\frac{f_1}{f_2^2}$ transforms to 
$\frac{f_2}{f_1 f_4}$ which is the generating function for $pod(n)$, the number of partitions of $n$ wherein the odd parts must be distinct (and the even parts are unrestricted) \cite{HSpod, RSpod}.

Therefore
\begin{align}\label{idp2}
MO(2, t; N)=\sum_{k+\binom{n+1}{2}=N}pod(k)\cdot c_n(2, t)
\end{align}
where
$$
c_n(2, t)=\binom{n+t}{2t}.
$$
In \cite{HSpod}, the authors prove that, for all $k\geq 0$, 
\begin{equation}
\label{pod_27_26_mod3}
pod(27k+26) \equiv 0 \pmod{3}.
\end{equation}
We now use this fact to prove the following infinite families of congruences modulo 3.

\begin{theorem} \label{MO2}  For all $J\geq 0$ and all $N\geq 0$,
\begin{align*}
MO(2, 9J+4; 27N+9) &\equiv 0 \pmod{3},\\
MO(2, 9J+7; 27N) &\equiv 0 \pmod{3}.
\end{align*}
\end{theorem}

\begin{proof} Notice that, if $t=9J+4$ and $n\not\equiv 4 \pmod{9}$ then
$c_n(2, t)\equiv 0\pmod{3}$.
Therefore, by \eqref{idp2},
$$MO(2, 9J+4; 27N+9) \equiv \sum_{k+\binom{n+1}{2}=27N+9}(-1)^k pod(k)\cdot c_n(2, 9J+4)
\equiv 0 \pmod{3}$$
because $\binom{n+1}{2}\equiv 10 \pmod{27}$ 
when $n\equiv 4 \pmod{9}$, which implies $k\equiv -10+9\equiv 26 \pmod{27}$  
and  $pod(k)\equiv 0 \pmod{3}$ by \eqref{pod_27_26_mod3}.

Similarly, if $t=9J+7$ and $n\not\equiv 1,7 \pmod{9}$ then
$c_n(2, t)\equiv 0\pmod{3}$.
Therefore,  by \eqref{idp2},
$$MO(2, 9J+7; 27N) \equiv \sum_{k+\binom{n+t}{2t}=27N}(-1)^k pod(k)\cdot c_n(2, 9J+7)\equiv 0 \pmod{3}$$
because $\binom{n+1}{2}\equiv 1 \pmod{27}$ 
when $n\equiv 1,7 \pmod{9}$, which implies $k\equiv -1\equiv 26 \pmod{27}$  
and $pod(k)\equiv 0 \pmod{3}$ by \eqref{pod_27_26_mod3}.
\end{proof}

In \cite[Theorem 1.1]{CuiGuMa}, the authors proved
that, for all $k\geq 0$,
\begin{equation}\label{pod172mod5}
pod(625k+172) \equiv 0 \pmod{5},
\end{equation}
and
\begin{equation}\label{pod297mod5}
pod(625k+297) \equiv 0 \pmod{5}.
\end{equation}
This now allows us to prove two families of corresponding congruences modulo 5.

\begin{theorem} \label{MO25}  For all $J\geq 0$ and all $N\geq 0$,
\begin{align}
MO(2, 25J+12; 625N+250) &\equiv 0 \pmod{5}, \label{MO2_250}\\
MO(2, 25J+12; 625N+375) &\equiv 0 \pmod{5}. \label{MO2_375}
\end{align}
\end{theorem}

\begin{proof}
Let $p$ be an odd prime. If $t\equiv \frac{p^2-1}{2} \pmod{p^2}$ and $n\not\equiv \frac{p^2-1}{2} \pmod{p^2}$ then
$\binom{n+t}{2t}\equiv 0\pmod{p}$. Moreover, if $n\equiv \frac{p^2-1}{2} \pmod{p^2}$
then $\binom{n+1}{2}\equiv \frac{p^4-1}{8} \pmod{p^4}.$

Hence, by \eqref{idp2}, if $p=5$ then 
\begin{align*}
MO(2, &25J+12; 625N+250)\\
&= \sum_{k+\binom{n+1}{2}=625N+250}(-1)^k pod(k)\cdot c_n(2, 25J+12)\equiv 0 \pmod{5}
\end{align*}
because $\binom{n+1}{2}\equiv 78 \pmod{625}$ 
when $n\equiv 12 \pmod{25}$, which implies $k\equiv -78+250=172 \pmod{625}$.  
Thanks to \eqref{pod172mod5}, we see that \eqref{MO2_250} follows. 

Similarly, again by \eqref{idp2}, 
\begin{align*}
MO(2, &25J+12; 625N+375)\\
&= \sum_{k+\binom{n+1}{2}=625N+375}(-1)^k pod(k)\cdot c_n(2, 25J+12)\equiv 0 \pmod{5}
\end{align*}
because $k\equiv -78+375=297$, and thanks to \eqref{pod297mod5}, we see that \eqref{MO2_375} follows.
\end{proof}

\section{The Case $a=-2$}

As we open this section, we define the family of functions $p_{-r}(n)$ via their generating functions:  
$$\sum_{n\geq0}p_{-r}(n)q^n=\prod_{k\geq1}\frac1{(1-q^k)^r}.$$
From the table given in the introduction we have that
$$\sum_{k\geq 0} b_k(-2) q^k=\frac{1}{f_1^3}=\sum_{k\geq 0} \p3(k) q^{k}.$$
Therefore
\begin{align}\label{idm2}
MO(-2, t; N)=\sum_{k+\binom{n+1}{2}=N}\p3(k)\cdot c_n(-2, t)
\end{align}
where
$$c_n(-2, t)=(-1)^{n+t} \frac{2n+1}{2t+1} \binom{n+t}{2t}.$$

For our purposes, we wish to focus attention on $p_{-3}(n)$.  We have the following Ramanujan--like congruences.  

\begin{theorem}[\cite{KO}, Theorem 4]
For all $N\geq 0$, 
\begin{equation}\label{p311}
p_{-3}(11N+7)\equiv 0 \pmod{11}.
\end{equation}
\end{theorem}

\begin{theorem}[\cite{SB}, Theorem 1.1 (iv)]
For all $N\geq 0$, 
\begin{equation}
p_{-3}(25N+22)\equiv 0 \pmod{5}.\label{p35}
\end{equation}
\end{theorem}

\begin{theorem}[\cite{SB}, Theorem 1.2 (v)]
For all $N\geq 0$ and $M=2,4,5,6$, 
\begin{equation}
p_{-3}(49N+7M+1)\equiv 0 \pmod{7}.\label{p37}
\end{equation}
\end{theorem}

We now use Theorem \ref{p311}, Theorem \ref{p35}, and Theorem \ref{p37} to prove the following new families of congruences
(in addition to the arithmetic properties that were mentioned in Section \ref{sec:known_properties} that correspond to the value $a=2$). 

\begin{theorem} \label{MOm211}  For all $J\geq 0$ and $N\geq 0$,
\begin{align*}
&MO(-2, 11J+4; 11N+6)\equiv 0 \pmod{11}.
\end{align*}
\end{theorem}

\begin{proof} Note that for any odd prime $p$,
if  $t\equiv \frac{p-3}{2}\pmod{p}$ and $c_n(-2, t)\not \equiv 0 \pmod{p}$ then  $n \equiv \frac{p-3}{2},\frac{p+1}{2} \pmod{p}$
which implies that $\binom{n+1}{2}\equiv \binom{\frac{p-1}{2}}{2} \pmod{p}$.

Hence, for $p=11$, we find $\binom{n+1}{2} \equiv 10 \pmod{11}$, and by \eqref{idm2},
$$MO(-2, 11J+4; 11N+6)=\sum_{k+\binom{n+1}{2}=11N+6}\p3(k)\cdot c_n(-2, 11J+4)\equiv 0 \pmod{11}$$
because $k\equiv 6-10\equiv 7 \pmod{11}$, and, by \eqref{p311}, $\p3(k)\equiv 0 \pmod{11}.$
\end{proof}

\medskip 

\begin{remark}
The above strategy works also for the congruence
$$MO(-2, 11J+10; 11N+7)\equiv 0 \pmod{11}$$
proved in \cite[Theorem 1.5]{AOS}. Indeed, for any odd prime $p$,
$$\text{if } t\equiv p-1\pmod{11}\text{ and } c_n(-2, t)\not \equiv 0 \pmod{p} 
\text{ then }
n \equiv 0,p-1 \pmod{p}$$ 
which implies that $\binom{n+1}{2}\equiv 0 \pmod{p}.$

Hence, for $p=11$, we find $\binom{n+1}{2} \equiv 0\pmod{11}$, and by \eqref{idm2},
$$MO(-2, 11J+10; 11N+7)=\sum_{k+\binom{n+1}{2}=11N+7}\p3(k)\cdot c_n(-2, 11J+7)\equiv 0 \pmod{11}$$
because $k\equiv 7 \pmod{11}$, and, by \eqref{p311}, $\p3(k)\equiv 0 \pmod{11}.$
\end{remark}

\begin{theorem} \label{MOm225} For all $J\geq 0$ and $N\geq 0$,
\begin{align*}
&MO(-2, 25J+11; 25N+13)\equiv 0 \pmod{5},\\
&MO(-2, 25J+24; 25N+22)\equiv 0 \pmod{5},
\end{align*}
and for $M=0, 1, 2, 5$, 
$$MO(-2, 49J+23; 49N+7M+11)\equiv 0 \pmod{7},$$
whereas, for $M=2, 4, 5, 6$, 
$$MO(-2, 49J+48; 49N+7M+1)\equiv 0 \pmod{7}.$$
\end{theorem}

\begin{proof} Note that if $t\equiv \frac{p^2-3}{2}\pmod{p^2}$ and $c_n(-2, t)\not \equiv 0 \pmod{p}$ then 
$n \equiv \frac{p^2-3}{2},\frac{p^2+1}{2} \pmod{p^2}$
which implies that $\binom{n+1}{2}\equiv \binom{\frac{p^2-1}{2}}{2} \pmod{p^2}.$

\noindent Thus, for $p=5$,
$$\text{if } t\equiv 11\pmod{25}\text{ and } c_n(-2, t)\not \equiv 0 \pmod{5} 
\text{ then } n \equiv 11,13 \pmod{25} $$ 
which implies that $\binom{n+1}{2}\equiv 16 \pmod{25}.$
Therefore, by \eqref{idm2},  we find
\begin{align*}
MO(-2, &25J+11; 25N+13)\\
&=\sum_{k+\binom{n+1}{2}=25N+13}p_{-3}(k)\cdot c_n(-2, 25J+11)\equiv 0\pmod{5}
\end{align*}
because $k\equiv 13-16\equiv 22 \pmod{25}$, and by \eqref{p35}, $p_{-3}(k)\equiv 0\pmod{5}$.

\noindent For $p=7$,
$$\text{if } t\equiv 23\pmod{49}\text{ and } c_n(-2, t)\not \equiv 0 \pmod{7} 
\text{ then }
n \equiv 23, 25 \pmod{49}$$ 
which implies that $\binom{n+1}{2}\equiv 31 \pmod{49}.$
Hence, by \eqref{idm2}, for $M=0, 1, 2, 5$, we obtain
\begin{align*}
MO(-2, &49J+23; 49N+7M+11)\\
&=\sum_{k+\binom{n+1}{2}=49N+7M+11}p_{-3}(k)\cdot c_n(-2, 49J+23)\equiv 0\pmod{7}
\end{align*}
because $k\equiv 7M+11-31\equiv 7(M+4)+1 \pmod{49}$, and by \eqref{p37}, $p_{-3}(k)\equiv 0\pmod{7}$.\medskip

Furthermore, if $t\equiv p^2-1\pmod{p^2}$ and $c_n(-2, t)\not \equiv 0 \pmod{p}$ then $n \equiv 0,p^2-1 \pmod{p^2}$ 
which implies that $\binom{n+1}{2}\equiv 0 \pmod{p^2}.$

\noindent Hence, for $p=5$,
$$\text{if } t\equiv 24\pmod{25}\text{ and } c_n(-2, t)\not \equiv 0 
\pmod{25} \text{ then }
n \equiv 0, 24 \pmod{25}$$ 
which implies that $\binom{n+1}{2}\equiv 0 \pmod{25}.$
Therefore, by \eqref{idm2}, for $M=2, 4, 5, 6$, we find
\begin{align*}
MO(-2, &25J+24; 25N+22)\\
&=\sum_{k+\binom{n+1}{2}=25N+22}p_{-3}(k)\cdot c_n(-2, 25J+24)\equiv 0\pmod{5}
\end{align*}
because $k\equiv 22 \pmod{25}$, and by \eqref{p35}, $p_{-3}(k)\equiv 0\pmod{5}$.

\noindent For $p=7$,
$$\text{if } t\equiv 48\pmod{49}\text{ and } c_n(-2, t)\not \equiv 0 \pmod{49} 
\text{ then } n \equiv 0, 48 \pmod{49}$$ 
which implies that $\binom{n+1}{2}\equiv 0 \pmod{49}$. 
Thus, by \eqref{idm2}, we get
\begin{align*}
MO(-2, &49J+48; 49N+7M+1)\\
&=\sum_{k+\binom{n+1}{2}=49N+7M+1}p_{-3}(k)\cdot c_n(-2, 49J+48)\equiv 0\pmod{7}
\end{align*}
because $k\equiv 7M+1 \pmod{49}$, and by \eqref{p37}, $p_{-3}(k)\equiv 0\pmod{7}$.
\end{proof}

In preparation for proving Theorem \ref{mod9} below, we highlight the following congruence modulo 3. 

\begin{lemma}
\label{3n+1mod3}
For all $k\geq 0$, $p_{-3}(3k+1) \equiv 0 \pmod{3}$
and $p_{-3}(3k+2) \equiv 0 \pmod{3}.$
\end{lemma}
\begin{proof}
It suffices to note that
$$\sum_{n\geq0}p_{-3}(n)x^n=\prod_{k\geq1}\frac1{(1-x^k)^3} \equiv \prod_{k\geq1}\frac1{(1-x^{3k})} \pmod{3}$$ 
and this last product is a function of $q^3$. 
\end{proof}

\begin{theorem}\label{mod9} For all $J\geq 0$ and $N\geq 0$,
\begin{align*}
&MO(-2, 3J; 3N+2) \equiv 0 \pmod{9},\\  
&MO(-2, 3J+2; 3N+2) \equiv 0 \pmod{9}.
\end{align*}
\end{theorem}

\begin{proof}
By \eqref{idm2}, we have that for $t\equiv 0,2 \pmod{3}$, 
$$MO(-2, t; 3N+2)=\sum_{k+\binom{n+1}{2}=3N+2} p_{-3}(k)\cdot c_n(-2, t)\equiv 0\pmod{3}$$ 
because if $n\equiv 0,2\pmod{3}$ then $\binom{n+1}{2}\equiv 0\pmod{3}$ which implies that $k\equiv 2 \pmod{3}$ and $p_{-3}(3k+2)\equiv 0 \pmod{9}$ (see \cite{Gandhi}),
otherwise if $n\equiv 1\pmod{3}$ then $\binom{n+1}{2}\equiv 1\pmod{3}$ which implies that $k\equiv 1 \pmod{3}$, $p_{-3}(3k+1)\equiv 0 \pmod{3}$ (from Lemma \ref{3n+1mod3}) and 
$3$ divides $c_n(-2, t)$ (indeed $2n+1\equiv 0 \pmod{3}$ and $2t+1\not \equiv 0 \pmod{3})$.
\end{proof}

\section{The case $a=0$}

From the table given in the introduction we have that
$$\sum_{k\geq 0} b_k(0) q^k=\frac{f_2}{f_1 f_4}=\sum_{k\geq 0} pod(k) q^{k}.$$ 
Therefore
\begin{equation}\label{id0}
MO(0, t; N)=\sum_{k+\binom{n+1}{2}=N}pod(k)\cdot c_n(0, t)
\end{equation}
where
$$c_n(0, t)=(-1)^{n+\lfloor (n+t)/2\rfloor} \binom{\lfloor(n+t)/2\rfloor}{t}.$$

Although there do not appear to be many Ramanujan--like congruences satisfied by $MO(0, t; N)$, we can highlight the following infinite family of congruences modulo 3  which follows from \eqref{pod_27_26_mod3}.
 
\begin{theorem} \label{MO0}  For all $J\geq 0$ and all $N\geq 0$,
\begin{align*}
MO(0, 27J+26; 27N+26) &\equiv 0 \pmod{3}.
\end{align*}
\end{theorem}
\begin{proof} If $t=27J+26$ and $n\not\equiv 0, 26 \pmod{27}$ then
$c_n(0, t)\equiv 0\pmod{3}$.
Therefore, from \eqref{id0}, 
$$MO(0, 27J+26; 27N+26) \equiv \sum_{k+\binom{n+1}{2}=27N+26}pod(k)\cdot c_n(0, t)\equiv 0 \pmod{3}$$
because $\binom{n+1}{2}\equiv 0 \pmod{27}$ 
when $n\equiv 0, 26 \pmod{27}$, which implies $k\equiv 26 \pmod{27}$  
and, thanks to \eqref{pod_27_26_mod3}, 
$pod(k)\equiv 0 \pmod{3}$ for such values of $k$.
\end{proof} 

\section{Appendix: $\B3(15n+7)\equiv 0\pmod{5}$}

We close this paper by providing an elementary proof that, for all $n\geq 0$, 
$$
\B3(15n+7)\equiv 0\pmod{5}.
$$  
(This fact was used in the proof of Theorem \ref{MOm1m5} above.)

From Ballantine and Merca \cite{Ballantine-Merca}, we know that
$$\sum_{n\geq 0}\B3(n)q^n=\frac{f_2f_3}{f_1^2 f_6}=\sum_{n\geq 0}\op(n)q^n \cdot \frac{f_3}{f_6}$$
where $\op(n)$ is the number of overpartitions of $n$ \cite{CL, HSover}.
Recalling that 
$$(-q;-q)_{\infty}=\frac{f_2^3}{f_1f_4},
$$ 
it follows that
$$\sum_{n\geq 0}\B3(n)(-q)^n=\sum_{n\geq 0}\op(n)(-q)^n \cdot \frac{\frac{f_6^3}{f_3f_{12}}}{f_6}=\sum_{n\geq 0}\op(n)(-q)^n \cdot \frac{f_6^2}{f_3f_{12}}.$$
Therefore
\begin{align*}
\sum_{n\geq 0}\B3(3n+1)(-q)^n&=\sum_{n\geq 0}\op(3n+1)(-q)^n \cdot \frac{f_2^2}{f_1f_{4}} \\
&=2\frac{f_1^6 f_4^6f_6^9}{f_2^{16}f_3^3f_{12}^3}\\
&=2\left(\frac{f_1 f_4f_6^2}{f_2^{3}f_3f_{12}}\right)^5\cdot \frac{f_1 f_4}{f_2}
\cdot\frac{f_3^2 f_{12}^2}{f_6}
\end{align*}
where we applied  \cite[see the formula at the end of page 4]{HSover} 
$$\sum_{n\geq 0}\op(3n+1)(-q)^n =2\frac{\varphi(q^3)^2}{\varphi(q)^4}\cdot X(q)=
2\frac{f_1^7 f_4^7f_6^9}{f_2^{18}f_3^3f_{12}^3},$$
with
$$\varphi(q)=\sum_{n=-\infty}^{\infty} q^{n^2}=\frac{f_2^5}{f_1^2 f_4^2}\quad\text{and}\quad
X(q)=\sum_{n=-\infty}^{\infty} q^{3n^2+2n}=\frac{f_2^2 f_3 f_{12}}{f_1f_4 f_6}.$$
Hence
$$\sum_{n\geq 0}\B3(3n+1)(-q)^n\equiv F(q^5) \cdot \frac{f_1 f_4}{f_2}
\cdot\frac{f_3^2 f_{12}^2}{f_6}\pmod{5}$$
where
$$F(q^5)=2\frac{f_5f_{20}f_{30}^2}{f_{10}^3f_{15}f_{60}}\equiv 2\left(\frac{f_1 f_4f_6^2}{f_2^{3}f_3f_{12}}\right)^5\pmod{5}.$$
Since $15n+7=3(5n+2)+1$, it suffices to show that the coefficient of $q^{5n+2}$ in 
$\frac{f_1 f_4}{f_2}
\cdot\frac{f_3^2 f_{12}^2}{f_6}$ is divisible by $5$. 
Indeed, from \cite[Theorem 1.1]{LO}, we see that 
$$\frac{f_1 f_4}{f_2}=\sum_{k\geq 0}(-q)^{\frac{k(k+1)}{2}}$$
and
$$\frac{f_3^2 f_{12}^2}{f_6}=\sum_{j\geq 1}\legendre{j}{3}jq^{j^2-1}.$$
Thus
$$[q^{5n+2}]\frac{f_1 f_4}{f_2}
\cdot\frac{f_3^2 f_{12}^2}{f_6}=\sum_{\frac{k(k+1)}{2}+j^2-1=5n+2}\left((-1)^{\frac{k(k+1)}{2}}\legendre{j}{3}j\right)\equiv 0\pmod{5}$$
because $\frac{k(k+1)}{2}+j^2-1=5n+2$ implies $k(k+1)+2j^2\equiv 1 \pmod{5}$.
Notice that, modulo 5, $k(k+1)\in\{0,1,2\}$ and $2j^2\in\{0,2,3\}$. Thus the previous congruence is satisfied only when $k(k+1)\equiv 1$ and $2j^2\equiv 0$. Hence we may conclude that $j$ is divisible by $5$, and this yields our result.


\begin{thebibliography}{99}

\bibitem{AAT1} T. Amdeberhan, G. E. Andrews, and R. Tauraso,
\emph{ Extensions of MacMahon's sums of divisors}, 
Res. Math. Sci. {\bf 11} (2024), Article 8.

\bibitem{AAT2} T. Amdeberhan, G. E. Andrews, and R. Tauraso,
\emph{Further study on MacMahon-type sums of divisors},
Res. Number Theory {\bf 11} (2025), Article 19.


\bibitem{AOS} T. Amdeberhan, K. Ono, and A. Singh,
\emph{MacMahon's sums-of-divisors and allied $q$-series},
Adv. Math. \textbf{452} (2024), 109820.

\bibitem{Andrews} G. E. Andrews, 
\emph{Theory of partitions}, Cambridge Univ. Press (1998).

\bibitem{Andrews-Rose} G. E. Andrews and S. C. F. Rose,
\emph{MacMahon's sum-of-divisors functions, Chebyshev polynomials, and quasi-modular forms}, 
J. Reine Angew. Math. \textbf{676} (2013), 97--103.

\bibitem{Ballantine-Merca}  C. Ballantine and M. Merca, 
\emph{Almost 3-regular overpartitions},
Ramanujan J. \textbf{58} (2022), 957--971. 

\bibitem{CL}
S. Corteel and J. Lovejoy, 
\emph{Overpartitions},
Trans. Amer. Math. Soc. \textbf{356}, no. 4 (2004), 1623--1635.

\bibitem{CuiGuMa} S.-P. Cui, W. X. Gu, and Z. S. Ma,
\emph{Congruences for partitions with odd parts distinct modulo 5},
Int. J. Number Theory \textbf{11} (2015), 2151--2159. 

\bibitem{Fine}  N. J. Fine, \emph{Binomial coefficients modulo a prime}
Am. Math. Mon. \textbf{54} (1947), 589--592.

\bibitem{Gandhi}  J. M. Gandhi, \emph{Congruences for $p_r(n)$ and Ramanujan's $\tau$-functions}
Am. Math. Mon. \textbf{70} (1963), 265--274. 

\bibitem{HSpod}  M. D. Hirschhorn and J. A. Sellers, \emph{Arithmetic properties of partitions with odd parts pistinct}, Ramanujan J.  \textbf{22}, no. 3
(2010), 273--284.

\bibitem{HSover} M. D. Hirschhorn and J. A. Sellers,
\emph{Arithmetic Properties for Overpartitions}, J. Comb. Math. Comb. Comput. (JCMCC) \textbf{53} (2005), 65--73.

\bibitem{KO} I. Kiming and J. Olsson, \emph{Congruences like Ramanujan’s for powers of the partition function}, Arch. Math. \textbf{59} (1992), 348--360.

\bibitem{MacMahon} P. A. ~MacMahon, 
\emph{Divisors of Numbers and their Continuations in the Theory of Partitions},
Proc. London Math. Soc. (2) \textbf{19} (1920), no.1, 75-113 
[also in Percy Alexander MacMahon Collected Papers, Vol.2, pp. 303--341 (ed. G.E. Andrews), MIT Press, Cambridge, 1986].

\bibitem{LO} R. J. Lemke Oliver, 
\emph{Eta--quotients and theta functions}, 
Adv. Math. {\bf 241} (2013), 1--17.

\bibitem{OS} K. Ono and A. Singh, 
\emph{Remarks on MacMahon's $q$-series},
J. Combin. Theory Ser. A \textbf{207} (2024), 105921.

\bibitem{RSpod}
S. Radu and J. A. Sellers, \emph{Congruence properties modulo 5 and 7 for the $pod$ function}, Int. J. Number Theory \textbf{7}, no.
8 (2011), 2249--2259.

\bibitem{R21} S. Ramanujan, \emph{Congruence properties of partitions}, 
Math. Z. \textbf{9} (1921), 147--153. 

\bibitem{SB} 
N. Saikia and C. Boruah, \emph{General congruences modulo 5 and 7 for colour partitions}, J. Anal. \textbf{29}, no. 3 (2021), 917--926.

\end{thebibliography}
\end{document}